\documentclass[a4paper,10pt,twoside]{amsart}

\usepackage[latin1]{inputenc}
\usepackage{amsmath, amsfonts, amssymb,amsthm}
\usepackage[mathscr]{eucal} %Proporciona el comando
\usepackage[pdftex]{graphicx}
\usepackage{color}
\usepackage{multicol}
\usepackage[T1]{fontenc}
\usepackage[sc]{mathpazo}
\usepackage[all]{xy}
\usepackage{hyperref}
\usepackage[flenq]{nccmath}

\usepackage {anysize}
\marginsize{2cm}{2cm}{2cm}{3cm}

\usepackage{enumerate}

\usepackage[pagewise]{lineno}

\definecolor{alert}{rgb}{0.8,0,0}

\newcommand{\R}{\mathbb{R}}

\newcommand{\s}{\mathbb{S}}
\newcommand{\h}{\mathbb{H}}

\newcommand{\M}{\mathbb{M}}

\newcommand{\grad}{\mathrm{grad}}

\newcommand{\olz}{\overline{z}}

\newcommand{\mrf}{\mathbb{R}\times_f\mathbb{M}^2(\kappa)}
\newcommand{\rrf}{\mathbb{R} \times_f \mathbb{R}^{2}}

\newcommand{\abs}[1]{\left\lvert #1 \right\rvert}

\newtheorem{theorem}{Theorem}[section]

\newtheorem{lemma}[theorem]{Lemma}

\theoremstyle{definition}
  \newtheorem{definition}[theorem]{Definition}

\theoremstyle{remark}
\newtheorem{remark}[theorem]{Remark}

\numberwithin{equation}{section}

\title[Height Estimates for surfaces with some constant curvature in \texorpdfstring{$\rrf$}{rrf}]{HEIGHT ESTIMATES FOR SURFACES WITH SOME CONSTANT CURVATURE IN \texorpdfstring{$\rrf$}{rrf}}

\dedicatory{Dedicated to memory of Mikhail Armenovich Malakhaltsev.}

\author{Jairo Delgado}
\address{Universidad Nacional de Colombia, Sede Palmira, Valle del Cauca-Colombia}
\email{jaadelgadoos@unal.edu.co.}
\address{Jairo Delgado is presently a part-time Professor (Profesor hora c\'atedra) at the Universidad del Valle, Sede Buga}
\email{jairo.delgado@correounivalle.edu.co}

\author{Haimer A. Trejos}
\address{Universidade do Estado do Rio de Janeiro, Rua S\~{a}o Francisco Xavier, 524 - 6 andar - Maracan\~{a},  Rio de Janeiro, Brazil}
\email{alexander.serna@ime.uerj.br}

\author{Carlos Pe\~{n}afiel}
\address{ Instituto de Matem\'{a}tica, Universidade Federal de Rio de Janeiro, Rio de Janeiro, Brazil, 21941-909}
\email{penafiel@im.ufrj.br}


\begin{document}
%\linenumbers

\begin{abstract}

In this paper, we obtain the necessary equations in a conformal parameter induced by the first or second fundamental forms for a surface that is isometrically immersed in the warped product $\mrf$ where $\mathbb{M}^{2}(\kappa)$ denotes the complete, connected, simply connected, two-dimensional space form of constant curvature. The surface we will consider has either positive extrinsic curvature or positive mean curvature. In each case, we carry out some geometric applications to the theory of constant curvature surfaces immersed in $\rrf$ under certain conditions on the warping function $f$. Specifically, we derive height estimates for graph-type surfaces with either positive constant extrinsic curvature or positive constant mean curvature. In particular, we classify compact minimal graphs in such warped products. This article extends previous work on the study of constant curvature surfaces immersed in product spaces using conformal parameters, as well as the height estimates for constant curvature surfaces in the warped product $\rrf$.

 \end{abstract}

\maketitle

\section{Introduction} \label{intro}

The study of minimal surfaces and constant mean curvature surfaces is one of the most widely researched topics in differential geometry today, as many of their results characterize these classes of surfaces analytically, geometrically, or topologically. Additionally, these surfaces have the interesting property of being critical points of functionals that have been extensively studied in recent decades. Due to their nature as surfaces, we can also define special parameters that allow us to study these classes of surfaces in greater detail. For example, Heinz Hopf in \cite{Ho} used this idea to define a holomorphic quadratic differential, known as the Hopf differential, which has the property of being holomorphic when the surface has constant mean curvature. Moreover, the zeros of this differential coincide with the umbilical points of the surface. With these facts, Hopf was able to prove the well-known Hopf theorem, which characterizes any constant mean curvature surface immersed in $\mathbb{R}^{3}$
with genus zero as the round sphere. It is important to mention that the argument used in the proof of Hopf's theorem can be generalized analogously to 3-manifolds of constant curvature, as the proof fundamentally depends on the special parameters considered on the surface and the type of umbilical surfaces in these spaces.

Subsequently, by using the conformal structure of the surface, T. K. Milnor in \cite{TK1}, studied the theory of Codazzi pairs and defined holomorphic quadratic differentials in an abstract sense. This, in particular, allowed  the extension of all the work done by Hopf to study quadratic differentials that depend solely on intrinsic properties of the surface, enabling the study of surfaces with curvature constraints that are independent of the ambient space in which these surfaces are immersed. Specifically, these ideas were used by Aledo, Galvez, and Espinar in \cite{AEG2} to generalize Hopf's theorem in an intrinsic sense, thereby achieving various applications to the theory of surfaces in different ambient spaces, such as the uniqueness of spheres in certain geometric problems and height estimates for compact graphs in $\mathbb{R}^{3}$. Following this idea, in \cite{AEG} the aforementioned authors calculated the necessary equations in conformal parameters for a surface to be immersed in a product space $\mathbb{M}^{2}(\kappa) \times \mathbb{R}$. Thus, they were able to derive vertical height estimates for compact constant mean curvature graphs in these ambient spaces. This height estimate is optimal because they showed the existence of rotational spheres that achieve this height and for which the Abresch-Rosenberg differential vanishes (see \cite{AR}).
Besides, in \cite{EGR} wrote the necessary equations for a surface of positive extrinsic curvature immersed in the product spaces $\mathbb{M}^{2}(\kappa) \times \mathbb{R}$. With this, they were able to make several applications to the theory of surfaces in these spaces, among which they calculated vertical and horizontal height estimates for compact graphs with constant positive extrinsic curvature and characterized the complete surfaces of constant extrinsic curvature as rotational spheres. Additional works on the characterization of constant mean curvature surfaces and extrinsic curvature surfaces in homogeneous spaces using the conformal structure of the surface have also been explored in several references (see \cite{ER,ER2}).

Finally, by using methods different from the conformal structure of the surface, height estimates are made for surfaces immersed in more general spaces or under other curvature conditions. For example in \cite{GIR}, the authors have used the maximum principle to derive height estimates for surfaces immersed in warped products, assuming that the symmetric functions of the principal curvatures (also called k-curvatures) are constant. In particular, they provide height estimates for compact graphs with constant mean curvature. Similarly, in \cite{M}, the author derives height estimates for compact elliptic graphs that are special Weingarten surfaces in product spaces $\mathbb{M}^{2}(\kappa) \times \mathbb{R}$ , using a maximum principle for this class of surfaces.

In this article, we compute the necessary equations in conformal parameters induced by the first or second fundamental forms for a surface which is immersed in the warped product $\mrf$, by using the immersion equations already calculated in \cite{ES1} and presented in a more suitable way for us in \cite{CFP}. Subsequently, we apply these results to the theory of constant curvature surfaces in the warped product $\rrf$, such as height estimates and the classification of compact minimal graphs in this warped product. It is important to note that this work extends the estimates made in \cite{AEG} and \cite{EGR} to the warped products and those in \cite{FPS} and \cite{GIR} for a broader class of warping functions.

This paper is organized as follows: In Section 2, we state the definitions and basic facts about the warped product $\mrf$ that will be used in this article and the theory of surfaces in these spaces. In Section 3, we write the necessary equations in the conformal parameters induced by the first or the second fundamental forms, along with the assumptions required to perform these calculations. Finally, in Section 4, we carry out the aforementioned applications in the warped product $\rrf$. For example, we have the following main theorems.

\nonumber\begin{theorem} 
Let $\Sigma \looparrowright \mathbb{R} \times_{f} \mathbb{R}^{2} $  be a compact graph in a domain of $\mathbb{R}^{2}$ where $f$ is a non-negative warping function, such that the first derivative $f'$ is non-positive and the second derivative $f''$ is non-negative. Suppose that $\Sigma$  has a positive constant extrinsic curvature $K_{e}$ and the boundary $\partial \Sigma$ is a subset of the slice $\{ 0 \} \times \mathbb{R}^{2}$. Then, the height function $h: \Sigma \rightarrow \mathbb{R}$ satisfies
\begin{equation*}
    h(p) \leq \frac{e^{f(0)}}{\sqrt{K_{e}}}
\end{equation*}
for each $p \in \Sigma$.  
\end{theorem}

\begin{theorem}
Let $\Sigma \looparrowright \mathbb{R} \times_{f} \mathbb{R}^{2} $  be a compact graph on an domain of $\mathbb{R}^{2}$ where $f$ is a non-negative warping function, such that $f'$ is non-positive and $f''$ is non-negative. Moreover, suppose that $\Sigma$ has a positive constant mean curvature $H$ and that the boundary $\partial \Sigma$ is contained in the slice $\mathbb{R}^{2} \times \{0 \}$. Then, the height function $h: \Sigma \rightarrow \mathbb{R}$ satisfies
\begin{equation*}
h(p) \leq  \frac{e^{f(0)}}{H}
\end{equation*}
for each $p \in \Sigma$.
\end{theorem}

\section{Preliminaries}\label{preliminaries} 

In this section, we will introduce the necessary concepts for studying surfaces with constant curvature immersed in $\mrf$. To do so, we will follow references \cite{AEG2} \cite{CFP}, \cite{BO}, \cite{DC}, and the references therein.

\subsection{The warped product \texorpdfstring{$\mrf$}{mrf}}

Let $\mathbb{M}^2(\kappa)$ be the complete, simply connected, two-dimensional space form having constant Gaussian curvature $\kappa, \kappa\in \{-1,0,1\}$. Thus, we have the following cases:
\begin{description}
    \item[(a)] $\mathbb{M}^2(-1)=\h^2$, where $\h^2$ denotes the hyperbolic space.
    \item[(b)]  $\mathbb{M}^2(0)=\R^2$, where $\R^2$ denotes the Euclidean space.
    \item[(c)]  $\mathbb{M}^2(1)=\s^2$, where $\s^2$ is the Euclidean unit sphere.
\end{description}

Let $(x,y)$ be the coordinates of $\mathbb{M}^2(\kappa)$ which is endowed with the metric $d\sigma^2_\kappa=\lambda^2(\kappa)(dx^2+dy^2)$, where
\begin{linenomath*}$$\lambda(\kappa)=\left\{
\begin{array}{cc}
\dfrac{2}{1+\kappa(x^2+y^2)},& \ \ \ \textnormal{if $\kappa\neq0$}\\ 
 1,& \ \ \ \textnormal{if $\kappa=0$}
\end{array}
\right..
$$
\end{linenomath*}

Let us consider the 3-dimensional Riemannian product $\R\times\mathbb{M}^2(\kappa) $, here $\R$ denotes the real line with its standard metric $dt^2$, and the canonical projections on the first and second factors.  
\begin{linenomath*}$$\pi_1:\R \times \mathbb{M}^2(\kappa) \to \R \hspace{.3cm} \textnormal{and} \hspace{.3cm} \pi_2:\R\times\mathbb{M}^2(\kappa) \to \mathbb{M}^2(\kappa)$$ 
\end{linenomath*}
defined by $\pi_1(t,p)=t$ and $\pi_2(t,p)=p$, respectively. For each smooth real function $f: \R\to \R$, the warped product $\mrf=(\R\times\M^2(\kappa),\overline{g})$ is the manifold $\R\times\mathbb{M}^2(\kappa)$ which is endowed with the Riemannian metric
\begin{linenomath*}
\begin{equation*}
\overline{g}=dt^2+e^{2f(t)}d\sigma^2_\kappa.
\end{equation*}
\end{linenomath*}
From now on, we use the bilinear form $\langle\cdot,\cdot\rangle$ associated with the Riemannian metric $\overline{g}$. In addition, we denote by $\nabla^1$, $\nabla^2$, $\overline{\nabla}$ the Levi-Civita connection of $\R$, $\M^2(\kappa)$ and $\mrf$, respectively. The tensor curvatures of $\mrf$ and $\M^2(\kappa)$ are denoted by $\overline{R}$ and $R^2$, respectively. Given vector fields $A,B,C\in\chi(\mrf)$, we define the Riemannian curvature tensor by
\begin{linenomath*}
 $$\overline{R}(A,B)C=\overline{\nabla}_B\overline{\nabla}_AC-\overline{\nabla}_A\overline{\nabla}_BC +\overline{\nabla}_{[A,B]}C.$$
\end{linenomath*}

In order to give a relation between the vector fields on $\mathbb{M}, \R$ and on $\mrf$, we need some definitions.

\begin{definition}
Let $\phi:N_1\ \to\ N_2$ be a smooth mapping between two differentiable manifolds. Vector fields $X\in\chi(N_1)$  and $Y\in\chi(N_2)$ are $\phi$-related if
\begin{linenomath*}
$$d \phi(X(p))=Y(\phi(p)), \ \ \textnormal{for all }\ \ p\in\ N_1.$$
\end{linenomath*}
\end{definition} 
We recall that:
\begin{enumerate}
\item The  lift of a vector field $X\in\chi(\M^2(\kappa))$ is the unique vector field  $\overline{X}\in\chi(\mrf)$ that is $\pi_1$-related to $X$ and $\pi_2$-related to the  zero vector field on $\chi(\mathbb{R})$, where $\pi_1$ and $\pi_2$ are the canonical projections. We denote by $\mathfrak{L}(\M^2(\kappa))\subset\chi(\mrf)$ the set of all such lifts.
\item Similarly,  $\mathfrak{L}(\R)\subset\chi(\mrf)$ denote  the set of all lifts of vector fields $V\in\chi(\R)$ to $\mrf$, that is, the vector fields $\overline{V}\in\chi(\mrf)$ which are $\pi_2$-related to the vector field $V\in\chi(\R)$ and $\pi_1$-related to the zero vector field on $\chi(\mathbb{M}^2(\kappa))$.

\end{enumerate}

\begin{remark}
When there is no confusion, by abuse of notation, we will use the same notation for a vector field and for its horizontal lift. If necessary, we will use the over-bar to emphasize the lift of a vector field. Moreover, in order to simplify some expressions in $\mrf$, we denote the canonical vector field, which points in the positive directions of the $t$-axis, $x$-axis, $y$-axis by
\begin{displaymath}
\dfrac{\partial}{\partial_t}=\partial_t, \hspace{.5cm}  \dfrac{\partial}{\partial_x}=\partial_x, \hspace{.5cm} \dfrac{\partial}{\partial_y}=\partial_y. 
\end{displaymath}
and we let $\xi$ be the lift of the vector $\partial_t$, where $\partial_t$ denotes the unit tangent vector field to the real line $\R$. It is clear that the geometric structure of the warped product is contained in the data $\kappa$ and $f$.
\end{remark}

\subsection{Surfaces isometrically immersed in \texorpdfstring{$\mrf$}{mrf}}
Let $\varphi:\Sigma \looparrowright\mrf$ be an isometric  immersion from a complete, orientable surface $\Sigma$ into the warped product $\mrf$. For simplicity, we will study the properties of the immersion $\varphi$ as those of $\Sigma$ and we denote by $\Sigma$ the image $\varphi(\Sigma)$.  We denote by $g$ the pull-back of the metric $\overline{g}$  on $\Sigma$ and by $\nabla$ the Levi-Civita connection associated with $g$. 

 Let $N$ be a unit normal vector field on $\Sigma$ (from now on, we fix the orientation $N$ of $\Sigma$), we denote by $S$ the shape operator and by $B$ the second fundamental form of $\Sigma$. Then, we have $\overline{g}(B(X,Y),N)=g(X,SY)$ for all $X,Y\in\chi(\Sigma)$. The formulas of Gauss and Weingarten state, respectively, that for all $X,Y\in\chi(\Sigma)$

\begin{linenomath*}
\begin{equation}\label{e100}
\overline{\nabla}_XY=\nabla_XY+B(X,Y) \hspace{1cm} \textnormal{and} \hspace{1cm} \overline{\nabla}_XN=-SX.
\end{equation}
\end{linenomath*}

The Riemann curvature tensor $R$ of the surface $\Sigma$ is related to the second fundamental form and the curvature tensor $\overline{R}$ of the ambient space by means of the Gauss equation
\begin{linenomath*}
\begin{equation}\label{ac1}
R(X,Y)Z=(\overline{R}(X,Y)Z)^T- g(SY,Z)SX+g(SX,Z)SY,
\end{equation}
\end{linenomath*}
for all $X,Y,Z\in\chi(\Sigma)$, here $X^T$ describes the tangent component of $X$ on the surface $\Sigma$.

The Codazzi's equation is given by
\begin{linenomath*}
\begin{equation}\label{wc1}
\overline{R}(X,Y)N=\nabla_XSY-\nabla_YSX-S[X,Y],
\end{equation}
\end{linenomath*}
where $X,Y,Z\in\chi(\Sigma)$ and $N$ is the unit normal vector field compatible with the orientation of $\Sigma$.

On the other hand, we can decompose the vertical field $\xi$ over the tangent and normal components with respect to the surface $\Sigma$. That is, we can write
\begin{linenomath*}
\begin{equation}\label{e200}
\xi=T+\nu N
\end{equation}
\end{linenomath*}
where $T$ is the tangent projection of $\xi$ over $\Sigma$ and $\nu N$ is the orthogonal projection of $\xi$ over the unit normal vector field $N$.  The function $\nu$ is called \emph{ angle function}. With the aid of the vector field $T\in\chi(\Sigma)$ we define a tensor $\mathcal{T}:\chi^4(\Sigma)\to \R$ of order $4$ over $\Sigma$, given by 
%\begin{linenomath*}
%$$\mathcal{T}:\chi(\Sigma)^4\to \R$$
%\end{linenomath*}

\begin{equation*}
 \mathcal{T}_{X}^Y(Z,W)=g(X,Z)g(Y,T)g(W,T)- g(X,W)g(Y,T)g(Z,T),
\end{equation*}
for vector fields $X,Y,Z,W\in\chi(\Sigma)$.  With those definitions, we can state the necessary equations that must be satisfied on the surface $\Sigma$ for it to be isometrically immersed in $\mrf$.

\begin{theorem}[\cite{CFP}]\label{wd1}
Let $\Sigma$  be a simply connected Riemannian surface with metric $g$, Riemannian connection $\nabla$ and Riemann curvature tensor $R$.  Then,  for all $X,Y,Z,W\in\chi(\Sigma)$, we have 
\vspace{.3cm}

%\hspace{-.2cm} {\bf Gauss equation}

\vspace{-0.2cm}
\begin{equation}\tag{\bf Gauss equation}
\label{gauss}\mathcal{R}(X,Y,Z,W)  =  g(R(X,Y)Z,W)- g(SX,Z)g(SY,W) +g(SX,W)g(SY,Z).
\end{equation}

\begin{equation}\tag{\bf Codazzi equation}
\label{codazzi}\mathcal{S}(X,Y) =  \nabla_XSY-\nabla_YSX-S([X,Y]).
\end{equation}

Where

\begin{eqnarray*}
\mathcal{R}(X,Y,Z,W) & = & \left((f^\prime)^2-\kappa e^{-2f}\right)[g(X,W)g(Y,Z)-g(X,Z)g(Y,W)]- \\
& & -\left(f^{\prime\prime}+\kappa e^{-2f}\right) [\mathcal{T}_X^Y(Z,W)-\mathcal{T}_Y^X(Z,W)]
\end{eqnarray*}

and
\begin{equation*}
\mathcal{S}(X,Y)=-\nu\left(f^{\prime\prime}+\kappa e^{-2f}\right)[g(X,T)Y-g(Y,T)X].
\end{equation*}
Moreover, the following equations hold for the immersion

\begin{eqnarray}
\label{35} \vert T\vert^2+\nu^2&=&1\\[10pt]
\label{33} \nabla_XT-\nu SX & = & f^\prime[X-g(X,T)T] \\[15pt]
\label{34} g(SX,T)+d\nu(X) & = & -f^\prime\, \nu\, g(X,T)
\end{eqnarray} 
Where $f^\prime$ denotes the derivative of the warping function $f$ with respect to the variable $t$. 

\end{theorem}

\begin{remark}
The proof of the above theorem can be found in \cite{ES1}, where the authors have shown that these equations are, in fact, the compatibility equations for isometric immersions in $\mrf$.
\end{remark}

From now on, for the warping function $f$, we denote its derivatives with respect to the variable $t$ by $f^\prime$ and $f^{\prime\prime}$. While for other functions $h=h(m)$ we will denote its derivative with respect to the variable $m$ by $h_m$ and $h_{mm}$.

\section{Necessary equations for surfaces isommetrically immersed in \texorpdfstring{$\mrf$}{mrf}.}

In this section, we rewrite the equations in theorem \ref{wd1} with respect to a special parameter induced by the first or second fundamental form. To do so,  let us consider $\varphi:\Sigma \looparrowright\mrf$ an isometric immersion from a complete, orientable surface. We denote by $h:\Sigma\to\R$ the horizontal height (or just height) function of $\Sigma$, that is, $h=\pi_1(\varphi)$. Then, the horizontal vector field $\xi$ is the gradient of the function $t$. In fact, for any vertical field $X$, we have
\begin{eqnarray*}
\langle \grad (t),\xi\rangle & = & dt(\xi)=1, \\
\langle \grad (t),X\rangle & = & dt(X)=0, \hspace{.5cm} \forall X\in\chi(\mrf), \hspace{.5cm} \langle \xi,X\rangle=0.
\end{eqnarray*}
Thus, $\grad (t)$ is horizontal, since $\xi$ is a basis of the horizontal vector fields, we conclude $\grad (t)=\xi$. 

Moreover, we have decomposed the vector field $\xi$ in the tangent and normal components with respect to the surface $\Sigma$ as 
$\xi=T+\nu N$ where $T$ is the tangent projection of $\xi$ over $\Sigma$ and $\nu N$ is the orthogonal projection of $\xi$ over the unit normal vector field $N$. Hence, we deduce $T$ is the gradient of $h$ on $\Sigma$, since
\begin{equation*}
\langle\grad h,T\rangle=\langle\grad h,\xi-\nu N\rangle=\langle\grad h,\xi\rangle=1
\end{equation*}

\subsection{Equations for surfaces in \texorpdfstring{$mrf$}{mrf} with positive extrinsic  curvature.}

Let  $\varphi:\Sigma \looparrowright\mrf$ be a complete, orientable surface with positive extrinsic curvature $K_e$. Since $K_e>0$, the second fundamental form $II$ is a positive definite bilinear form (and positive definite for a suitable normal $N$). Moreover,  we can choose a conformal parameter $z$ for the second fundamental form $II$. Hence, the fundamental forms $I$ and $II$ can be written as
\begin{eqnarray}
\nonumber I & = & \langle d\varphi,d\varphi\rangle=Edz^2+2F\vert dz\vert^2+\overline{E}d\overline{z}^2 \\
\label{e1} II & = & -\langle d\varphi,dN\rangle=2\rho\vert dz\vert^2
\end{eqnarray}
Where $\rho$ is a real function such that $\rho>0$, $dz$ and $d\olz$ are the dual forms that correspond to vector fields $\partial_{z} = \frac{1}{2} (\partial_{x} - i \partial_{y})$ and   $\partial_{\olz} = \frac{1}{2}(\partial_{x} + i \partial_{y})$, respectively. If we denote by $H$ the mean curvature function of $\Sigma$ can be shown that (see \cite{AEG}, \cite{AEG2} )
\begin{eqnarray}
\label{e2} K_{e} & = & -\dfrac{\rho^2}{D} \\
\label{e3} H & = & \dfrac{K_e}{\rho}F
\end{eqnarray}
where $D = \vert E \vert^{2} - F^{2}$ and $H$ is the mean curvature of the surface. 
\begin{remark}
If $K_{e} >0$, then from equation \eqref{e2} we have $D < 0$.  Additionally, from equation \eqref{e3}, we can choose the orientation of $\Sigma$ such that the principal curvatures are positive, so that $H>0$ and we can assume $F>0$. 
\end{remark}

Before rewriting the necessary equations in the conformal parameter $z$ for the case of positive extrinsic curvature, for the sake of completeness, we show the following lemma,  which has already been proven in \cite{AEG} and \cite{AEG2}.

\begin{lemma}[\cite{AEG,AEG2}]\label{l1}
Let $\varphi:\Sigma \looparrowright\mrf$ be an immersion from a complete, orientable surface, with positive extrinsic curvature $K_e$. Then, in the conformal parameter $z$ for the second fundamental form $II$, the following equations hold:
\begin{eqnarray}
\label{e4} T &= & \dfrac{1}{D}[\alpha\partial_z+\overline{\alpha}\partial_{\overline{z}}] \\
\label{e5} \alpha & = & \overline{E}h_z-Fh_{\overline{z}}. \\
\label{e6} \langle T,T\rangle & = & \dfrac{1}{D}[\alpha h_z+\overline{\alpha}h_{\overline{z}}] \\
\label{e7} \langle T,T\rangle & = & \dfrac{1}{D}[Eh^2_{\overline{z}}+\overline{E}h^2_z-2Fh_zh_{\overline{z}}] \\
\label{e8} h_z & = & \dfrac{1}{D}[E\alpha+F\overline{\alpha}] \\
\label{e9} \vert h_z\vert^2 & = & \vert\vert T\vert\vert^2F+\dfrac{\vert\alpha\vert^2}{D}
\end{eqnarray}
Here $h_z$ denotes the derivative of the function $h$ with respect to the variable $z$, and so on.
\end{lemma}
\begin{proof}
Since $T\in\chi(\Sigma)$, we have $T=a\partial_z+b\partial_{\overline{z}}$, then
\begin{eqnarray*}
h_z & = & \langle T,\partial_z\rangle = aE+bF \\
h_{\overline{z}} & = & \langle T,\partial_{\overline{z}} \rangle = aF+b\overline{E}.
\end{eqnarray*}
Hence
\begin{displaymath}
\begin{pmatrix}
h_z\\
h_{\overline{z}}
\end{pmatrix} 
=
\begin{pmatrix}
E&F\\
F&\overline{E}
\end{pmatrix} 
\begin{pmatrix}
a\\
b
\end{pmatrix} 
\Rightarrow
\begin{pmatrix}
a\\
b
\end{pmatrix} 
=\dfrac{1}{D}
\begin{pmatrix}
\overline{E}&-F\\
-F&E
\end{pmatrix} 
\begin{pmatrix}
h_z\\
h_{\overline{z}}
\end{pmatrix}
\end{displaymath} 
And one gets 
\begin{equation*}
T=\dfrac{1}{D}[\alpha\partial_z+\overline{\alpha}\partial_{\overline{z}}], \hspace{5.7cm} \textnormal{Equation \eqref{e4}}
\end{equation*}
where
\begin{equation*}
\alpha=\overline{E}h_z-Fh_{\overline{z}}.  \hspace{6.3cm} \textnormal{Equation \eqref{e5}}
\end{equation*}
By definition, we have $T=\grad (h)$, so
\begin{equation*}
\langle T,T\rangle=\left\langle \grad (h),\dfrac{1}{D}[\alpha\partial_z+\overline{\alpha}\partial_{\overline{z}}] \right\rangle =\dfrac{1}{D}[\alpha h_z+\overline{\alpha}h_{\overline{z}}]. \hspace{1cm} \textnormal{Equation \eqref{e6}}
\end{equation*}
Moreover, from equation \eqref{e6}
\begin{eqnarray*}
\langle T,T\rangle & = & \dfrac{1}{D^2}[\alpha^2 E+2\alpha\overline{\alpha}F+\overline{\alpha}^2\overline{E}]  =  \dfrac{1}{D^2}D[Eh^2_{\overline{z}}+\overline{E}h^2_z-2Fh_zh_{\overline{z}}].
\end{eqnarray*}
That is
\begin{equation*}
\langle T,T\rangle=\dfrac{1}{D}[Eh^2_{\overline{z}}+\overline{E}h^2_z-2Fh_zh_{\overline{z}}] \hspace{3.8cm} \textnormal{Equation \eqref{e7}}
\end{equation*}
Now,
\begin{eqnarray*}
\dfrac{1}{D}[E\alpha+F\overline{\alpha}] & = & \dfrac{1}{D}[(\overline{E}h_z-Fh_{\overline{z}})E+(Eh_{\overline{z}}-Fh_z)F] \\
     & = & \dfrac{1}{D}\left[\vert E\vert^2 h_z-EFh_{\overline{z}}+EFh_{\overline{z}} -F^2h_z \right] \\
     & = & \dfrac{1}{D}\left[\vert E\vert^2-F^2 \right]h_z.
\end{eqnarray*}
and this implies the equation \eqref{e8}. Finally,  from the previous expressions, we deduce that
\begin{equation*}
\vert\vert T\vert\vert^2 F+\dfrac{\vert\alpha\vert^2}{D} = \dfrac{1}{D}\left[[\alpha h_z+\overline{\alpha}h_{\overline{z}}]F +\left(\vert E\vert^2h_zh_{\overline{z}} -\overline{E}Fh_z^2-EFh_{\overline{z}}^2 +F^2h_zh_{\overline{z}}\right)  \right]. 
\end{equation*}
Therefore,
\begin{equation*}
\vert\vert T\vert\vert^2 F+\dfrac{\vert\alpha\vert^2}{D} = \dfrac{1}{D}\left[\vert E\vert^2 h_zh_{\overline{z}}-F^2h_zh_{\overline{z}}  \right]=\vert h_z\vert^2, 
\end{equation*}
which is the equation \eqref{e9}.
\end{proof}

In the following theorem, we write the necessary equations in the conformal parameter $z$ induced by the second fundamental form $II$ when the surface $\Sigma$ has a positive extrinsic curvature and is immersed in the warped product  $\mrf$. In fact, this extends what was done in \cite{AEG2} and  \cite{EGR}.

\begin{lemma}\label{l2}
Let  $\varphi:\Sigma \looparrowright \mrf$ be an isometric immersion from a complete, orientable surface, with positive extrinsic curvature $K_e$. Then, for a conformal parameter $z$ for the second fundamental form, the following equations hold:
\begin{eqnarray}
\label{e10}  \textbf{Codazzi} \hspace{.5cm} \dfrac{\rho_{\overline{z}}}{\rho} & + & (\Gamma^1_{12}- \Gamma^2_{22})  =  \frac{\nu}{\rho}\alpha \left( f^{\prime\prime} +\kappa e^{-2f}  \right).\\
\label{e11}  \nu_z & = & \dfrac{\overline{\alpha}K_e}{\rho} -f^\prime\nu h_z. \\
\label{e12} \Gamma^1_{12} & = & -\dfrac{(K_{e})_{\overline{z}}}{4Ke} +\dfrac{\nu\alpha}{2\rho}\left( f^{\prime\prime} +\kappa e^{-2f} \right) \\
\label{e12.1} h_{zz} & = & \Gamma_{11}^{1} h_{z} + \Gamma^{2}_{11} h_{\overline{z}} + f^{\prime}E-f^{\prime} \big( h_{z} \big)^{2}.\\ 
\label{e13} h_{z\overline{z}} & = & \rho \nu \Big( 1 - \frac{( f^{\prime\prime} +\kappa e^{-2f})(1-\nu^{2})}{2 K_{e}} \Big) - \frac{1}{4K_{e}} \Big( (K_{e})_{\overline{z}}h_{z} + (K_{e})_{z}h_{\overline{z}} \Big) - f^{\prime} (F - h_{\overline{z}}h_{z}).\\   
\label{e14} \alpha_z & = & \alpha \frac{D_{z}}{2D} + \frac{1}{4K_{e}} \big( \overline{\alpha}(K_{e})_{\overline{z}} - \alpha (K_{e})_{z} \big) - \rho F \nu - f^{\prime} \big(\alpha h_{z}  - D \big).\\
\label{e15} \nu_{z \overline{z}} & = & \frac{1}{4 \rho} \Big( \overline{\alpha}(K_{e})_{\overline{z}} + \alpha (K_{e})_{z} \Big) - K_{e} F \nu - \frac{K_{e}f^{\prime}}{\rho} \Big(\alpha h_{z}  - D \Big) - f^{\prime}[\nu_{z} h_{\overline{z}} + \nu h_{z\overline{z}}] - f^{''}\nu \abs{h_{z}}^{2}
%\label{e13} \langle T,T\rangle & = & \dfrac{1}{D}[Eh^2_{\overline{z}}+\overline{E}h^2_z-2Fh_zh_{\overline{z}}] \\
%\label{e14} h_z & = & \dfrac{1}{D}[E\alpha+f\overline{\alpha}] \\
%\label{e15} \vert h_z\vert^2 & = & \vert\vert T\vert\vert^2F+\dfrac{\vert\alpha\vert^2}{D}
\end{eqnarray}
Here $h_z$ denotes the derivative of the function $h$ with respect to the variable $z$, and $f'$, $f''$ denote the derivatives of the function $f$ respect to $t$.
\end{lemma}
\begin{proof}
Take $X=\partial_z$ at equation \eqref{34}, hence
\begin{displaymath}
g\left( S\partial_z,\grad(h) \right)+ d\nu(\partial_z)=-f^\prime\nu g\left(\grad(h),\partial_z \right).
\end{displaymath}
So,
\begin{displaymath}
\nu_z=-g\left( S\partial_z,\dfrac{1}{D}\left(\alpha\partial_z +\overline{\alpha}\partial_{\overline{z}} \right) \right) -f^\prime\nu h_z
\end{displaymath}
\begin{displaymath}
\nu_z=-\dfrac{1}{D}\left[\alpha g\left( S\partial_z,\partial_z \right) +\overline{\alpha} g\left( S\partial_z,\partial_{\overline{z}} \right) \right] -f^\prime\nu h_z
\end{displaymath}
and we deduce that
\begin{displaymath}
\nu_z=-\dfrac{\rho}{D}\overline{\alpha}  -f^\prime\nu h_z =\dfrac{\overline{\alpha} K_e}{\rho} -f^\prime\nu h_z
\end{displaymath}
which shows equation \eqref{e11}.

On the other hand, from the Codazzi equation in theorem \ref{wd1}, if $X=\partial_z$, $Y=\partial_{\overline{z}}$ and taking the scalar product with $\partial_{\overline{z}}$, we get
\begin{eqnarray*}
\left\langle \mathrm{T}_S(\partial_z,\partial_{\overline{z}}), \partial_{\overline{z}} \right\rangle & = & \left\langle \nu(f^{\prime\prime}+\kappa e^{-2f})\left[ g\left(\partial_{\overline{z}},\grad(h) \right)\partial_z -g\left(\partial_{z},\grad(h) \right)\partial_{\overline{z}} \right], \partial_{\overline{z}}   \right\rangle \\
\left\langle \mathrm{T}_S(\partial_z,\partial_{\overline{z}}), \partial_{\overline{z}} \right\rangle & = & \nu(f^{\prime\prime}+\kappa e^{-2f})\left[ h_{\overline{z}}F-h_z\overline{E} \right]\\
\left\langle \mathrm{T}_S(\partial_z,\partial_{\overline{z}}), \partial_{\overline{z}} \right\rangle & = & -\nu(f^{\prime\prime}+\kappa e^{-2f})\alpha \\
\left\langle \mathrm{T}_S(\partial_{\overline{z}},\partial_z), \partial_{\overline{z}} \right\rangle & = & \nu\alpha(f^{\prime\prime}+\kappa e^{-2f}).
\end{eqnarray*}

Now, from \eqref{e1}
\begin{eqnarray*}
\nabla_{\partial_z} \partial_z & = & \Gamma_{11}^1 \partial_z+\Gamma_{11}^2 \partial_{\overline{z}}   \\
\nabla_{\partial_z} \partial_{\overline{z}} & = & \Gamma_{12}^1 \partial_z+\Gamma_{12}^2 \partial_{\overline{z}} +\rho N 
\end{eqnarray*}
which implies
\begin{displaymath}
\left\langle \nabla_{\partial_{\overline{z}}}\nabla_{\partial_z} \partial_z - \nabla_{\partial_{z}}\nabla_{\partial_{\overline{z}}} \partial_z, \partial_{\overline{z}} \right\rangle =\rho_{\overline{z}} +\rho\left( \Gamma^1_{12}- \Gamma^2_{22} \right),
\end{displaymath}
that is

\begin{displaymath}
\left\langle \mathrm{T}_S(\partial_{\overline{z}},\partial_z), \partial_{\overline{z}} \right\rangle =\rho_{\overline{z}} +\rho\left( \Gamma^1_{12}- \Gamma^2_{22} \right).
\end{displaymath}
We conclude that
\begin{equation}\label{e0}
\dfrac{\rho_{\overline{z}}}{\rho} +\left( \Gamma^1_{12}- \Gamma^2_{22} \right) = \frac{\nu}{\rho}\alpha(f^{\prime\prime}+\kappa e^{-2f})
\end{equation}
and this shows the equation \eqref{e10}.

On the other hand, by a straightforward computation in equation \eqref{e1}, we obtain
\begin{equation}\label{e12.2}
\Gamma^1_{12} +\Gamma^2_{22} = \dfrac{D_{\overline{z}}}{2D}
\end{equation}
The previous equation, together with equation \eqref{e0}, implies that
\begin{equation}\label{e00}
\dfrac{\rho_{\overline{z}}}{\rho} -\dfrac{D_{\overline{z}}}{2D} +2\Gamma^1_{12} = \nu\alpha(f^{\prime\prime}+\kappa e^{-2f}).
\end{equation}
Now we compute $(K_{e})_{\olz}$, if we take the derivative at \eqref{e2}, we obtain
\begin{equation*}
(K_e)_{\overline{z}} =\dfrac{-2\rho\rho_{\overline{z}}D+\rho^2 D_{\overline{z}}}{D^2}
\end{equation*}
Thus
\begin{equation*}
\dfrac{(K_e)_{\overline{z}}}{2K_e} =\left(\dfrac{-D}{2\rho^2} \right)\dfrac{-2\rho\rho_{\overline{z}}D+\rho^2 D_{\overline{z}}}{D^2}=\dfrac{\rho_{\overline{z}}}{\rho} -\dfrac{D_{\overline{z}}}{2D},
\end{equation*}
Consequently
\begin{equation*}
\dfrac{(K_e)_{\overline{z}}}{2K_e} =\dfrac{\rho_{\overline{z}}}{\rho} -\dfrac{D_{\overline{z}}}{2D}
\end{equation*}
Substituting the last equation into equation \eqref{e00}, we obtain the following equation
\begin{equation*}
\dfrac{(K_e)_{\overline{z}}}{2K_e} +2\Gamma^1_{12} = \nu\alpha(f^{\prime\prime}+\kappa e^{-2f})
\end{equation*}
which shows \eqref{e12}. Next, if $X=\partial_{z}$ in the equation \eqref{33}
\begin{eqnarray*}
    h_{zz} & = & \partial_{z} ( \langle \nabla h, \partial_{z} \rangle ) = \langle \nabla_{\partial_{z}} (\nabla h), \partial_{z} \rangle + \langle \nabla h, \nabla_{\partial_{z}} \partial_{z} \rangle \\
          & = & \langle \nu S(\partial_{z}) + f^{\prime}[\partial_{z}-\langle \partial_{z}, \nabla h \rangle \nabla h], \partial_{z} \rangle + h_{z} \Gamma^{1}_{11} + h_{\overline{z}}\Gamma^{2}_{11} \\
          & = & f^{\prime} E - f^{\prime} (h_{z})^{2} + + h_{z} \Gamma^{1}_{11} + h_{\overline{z}}\Gamma^{2}_{11}
\end{eqnarray*}
and this shows equation \eqref{e12.1}. Also, if   $X=\partial_{\overline{z}}$ in the equation \eqref{33} 

\begin{eqnarray*}
    h_{z \overline{z}} & = & \partial_{\overline{z}} ( \langle \nabla h, \partial_{z} \rangle ) = \langle \nabla_{\partial_{\overline{z}}} (\nabla h), \partial_{z} \rangle + \langle \nabla h, \nabla_{\partial_{\overline{z}}} \partial_{z} \rangle \\
    & = & \langle \nu S(\partial_{\overline{z}}) + f^{\prime}[\partial_{\overline{z}} - \langle \partial_{\overline{z}}, \nabla h \rangle \nabla h], \partial_{z} \rangle +  \langle \nabla h, \nabla_{\partial_{\overline{z}}} \partial_{z} \rangle \\ 
    & = & \rho \nu + f^{\prime} F - f^{\prime} h_{\overline{z}} h_{z} +  h_{z} \Gamma^{1}_{12} + h_{\overline{z}}\Gamma^{2}_{12}.\\ 
\end{eqnarray*}
Then, from equations \eqref{e6}, \eqref{e12} and the equality $\Gamma_{12}^{2} = \overline{\Gamma^{1}_{12}}$, We find the following

\begin{eqnarray*}
    h_{z \overline{z}} & = & \rho \nu + f^{\prime} F - f^{\prime} h_{\overline{z}} h_{z} - \frac{1}{4K_{e}} \Big( (K_{e})_{\overline{z}} h_{z} + (K_{e})_{z} h_{\overline{z}} \Big) + \frac{(f^{\prime \prime} + \kappa e^{-2f})\nu}{2 \rho} \Big( \alpha h_{z} + \overline{\alpha} h_{\overline{z}} \Big) \\
    & = & \rho \nu +  \frac{(f^{\prime \prime} + \kappa e^{-2f})\nu}{2 \rho} D(1-\nu^{2}) - \frac{1}{4K_{e}} \Big( (K_{e})_{\overline{z}} h_{z} + (K_{e})_{z} h_{\overline{z}} \Big) + f^{\prime} F - f^{\prime} h_{\overline{z}} h_{z}
\end{eqnarray*}
and finally from identity $D= \frac{-\rho^{2}}{K_{e}}$, 

\begin{equation*}
    h_{z \overline{z}} = \rho \nu \Big(1-  \frac{(f^{\prime \prime} + \kappa e^{-2f})(1-\nu^{2})}{2 K_{e}}\Big)  - \frac{1}{4K_{e}} \Big( (K_{e})_{\overline{z}} h_{z} + (K_{e})_{z} h_{\overline{z}} \Big) + f^{\prime} F - f^{\prime} h_{\overline{z}} h_{z}
\end{equation*}
and this the equation \eqref{e13} follows. Now, we compute $\alpha_{z}$:

\begin{eqnarray*}
\alpha_z & = & [\overline{E}h_{z} - F h_{\overline{z}}]_{z} = \overline{E}_{z}h_{z} + \overline{E}h_{zz} - F_{z}h_{\overline{z}} - F h_{z\overline{z}} \\
         & = & 2 \langle \nabla_{\partial_{z}} \partial_{\overline{z}} , \partial_{\overline{z}} \rangle h_{z} + \overline{E} h_{zz}- \langle \nabla_{\partial_{z}} \partial_{z} , \partial_{\overline{z}} \rangle h_{\overline{z}}- \langle \partial_{z}, \nabla_{\partial_{z}} \partial_{\overline{z}} \rangle h_{\overline{z}} - F h_{z \overline{z}} \\
         & = & 2 \big( \Gamma^{1}_{12}F + \Gamma^{2}_{12} \overline{E} \big) h_{z} +  \overline{E} h_{zz} - \big( \Gamma^{1}_{11} F + \Gamma^{2}_{11} \overline{E} \big) h_{\overline{z}} - \big( \Gamma^{1}_{12}E + \Gamma^{2}_{12} F \big) h_{\overline{z}}- F h_{z\overline{z}}.
\end{eqnarray*}
Then, from equation \eqref{e12.1}, 

\begin{eqnarray*}
    \alpha_{z} & = & 2 \big( \Gamma^{1}_{12}F + \Gamma^{2}_{12} \overline{E} \big) h_{z} +  \overline{E} \big(    \Gamma_{11}^{1} h_{z} + \Gamma^{2}_{11} h_{\overline{z}} + f^{\prime}E-f^{\prime} \big( h_{z} \big)^{2}  \big) - \big( \Gamma^{1}_{11} F + \Gamma^{2}_{11} \overline{E} \big) h_{\overline{z}} - \big( \Gamma^{1}_{12}E + \Gamma^{2}_{12} F \big) h_{\overline{z}} \\
    & - & F \big( \rho \nu + f^{\prime} F - f^{\prime} h_{\overline{z}} h_{z} +  h_{z} \Gamma^{1}_{12} + h_{\overline{z}}\Gamma^{2}_{12}   \big).\\
    & = & 2\Gamma_{12}^{2}(\overline{E}h_{z}-F h_{\overline{z}}) + \Gamma_{11}^{1} (\overline{E}h_{z} - F h_{\overline{z}}) + \Gamma_{12}^{1}(Fh_{z} - E h_{\overline{z}}) + f^{\prime}D - f^{\prime}h_{z}(\overline{E}h_{z}-Fh_{\overline{z}}) - F \rho \nu \\
    & = & \alpha (\Gamma^{2}_{12}+\Gamma^{1}_{11}) + \alpha \Gamma^{2}_{12} - \overline{\alpha}\Gamma^{1}_{12}+ f'D - f'\alpha h_{z} - F \rho \nu.
\end{eqnarray*}
By the one hand  $\Gamma^{2}_{12} = \overline{\Gamma^{1}_{12}}$ and from equation \eqref{e12}, we get $\alpha \Gamma^{2}_{12} - \overline{\alpha} \Gamma^{2}_{12} = \frac{1}{4K_{e}} \Big( \overline{\alpha} (K_{e})_{\overline{z}} - \alpha (K_{e})_{z} \Big)$. By the other hand, $\Gamma^{1}_{11} = \overline{\Gamma^{2}_{22}}$, then $\Gamma_{12}^{1} + \Gamma^{1}_{11} = \frac{D_{z}}{2D}$. Therefore

\begin{equation*}
    \alpha_{z} = \frac{\alpha D_{z}}{2D} + \frac{1}{4K_{e}} \Big( \overline{\alpha} (K_{e})_{\overline{z}} - \alpha (K_{e})_{z} \Big)  - \rho F \nu - f' (\alpha h_{z} - D)
\end{equation*}
and this shows \eqref{e14}. Finally we compute $\nu_{z \olz}$, from equations \eqref{e11} and \eqref{e14}, we have the following

\begin{eqnarray*}
     v_{z \overline{z}} & = &  (\nu_{\overline{z}})_{z} = \frac{K_{e}}{\rho} \alpha_{z} + (K_{e})_{z} \frac{\alpha}{\rho} - (K_{e}) \alpha \frac{\rho_{z}}{\rho^{2}} - f^{\prime}[\nu_{z} h_{\overline{z}} + \nu h_{z\overline{z}}] - f^{''}\nu \abs{h_{z}}^{2} \\
    & = & \frac{K_{e}}{\rho} \Big( \alpha \frac{D_{z}}{2D} + \frac{1}{4K_{e}} \big( \overline{\alpha}(K_{e})_{\overline{z}} - \alpha (K_{e})_{z} \big) - \rho F \nu + f^{\prime} \big(D - \alpha h_{z} \big) \Big) + (K_{e})_{z} \frac{\alpha}{\rho} - (K_{e}) \alpha \frac{\rho_{z}}{\rho^{2}} - f^{\prime}[\nu_{z} h_{\overline{z}} + \nu h_{z\overline{z}}] - f^{''}\nu \abs{h_{z}}^{2}\\
    & = & \frac{K_{e} \alpha}{\rho} \Big( \frac{D_{z}}{2D} - \frac{\rho_{z}}{\rho} \Big) + (K_{e})_{z} \frac{\alpha}{\rho} + \frac{1}{4 \rho} \Big( \overline{\alpha}(K_{e})_{\overline{z}} - \alpha (K_{e})_{z} \Big) - K_{e} F \nu + \frac{K_{e}f^{\prime}}{\rho} \Big(D - \alpha h_{z} \Big) - f^{\prime}[\nu_{z} h_{\overline{z}} + \nu h_{z\overline{z}}] - f^{''}\nu \abs{h_{z}}^{2}
\end{eqnarray*}
From $\frac{(K_e)_{z}}{2K_e} = \frac{\rho_{z}}{\rho} - \frac{D_{z}}{2D}$, we get

\begin{eqnarray*}
    \nu_{z \overline{z}} & = & -\frac{(K_e)_{z} \alpha}{2\rho} + (K_{e})_{z} \frac{\alpha}{\rho} + \frac{1}{4 \rho} \Big( \overline{\alpha}(K_{e})_{\overline{z}} - \alpha (K_{e})_{z} \Big) - K_{e} F \nu - \frac{K_{e}f^{\prime}}{\rho} \Big( \alpha h_{z} - D \Big) - f^{\prime}[\nu_{z} h_{\overline{z}} + \nu h_{z\overline{z}}] - f^{''}\nu \abs{h_{z}}^{2} \\
    & = & \frac{1}{4 \rho} \Big( \overline{\alpha}(K_{e})_{\overline{z}} + \alpha (K_{e})_{z} \Big) - K_{e} F \nu - \frac{K_{e}f^{\prime}}{\rho} \Big( \alpha h_{z} - D \Big) - f^{\prime}[\nu_{z} h_{\overline{z}} + \nu h_{z\overline{z}}] - f^{''}\nu \abs{h_{z}}^{2} 
\end{eqnarray*}
and this show \eqref{e15}.
\end{proof}

\subsection{Equations for surfaces in \texorpdfstring{$\mrf$}{mrf} with  positive constant mean curvature}
Let $\varphi:\Sigma \looparrowright \mrf$ be an isometric  immersion from a complete, orientable surface having a positive mean curvature $H>0$ for the given orientation $N$. In this case, we can choose a conformal parameter $z$ for the first fundamental form $I$. Hence, the fundamental forms $I$ and $II$ can be written as

\begin{eqnarray*}
\nonumber I & = & \langle d\varphi,d\varphi\rangle=\lambda\vert dz\vert^2 \\ 
\nonumber II & = & -\langle d\varphi,dN\rangle=pdz^2+\lambda H\vert dz\vert^2+\overline{p}d\overline{z}^2 
\end{eqnarray*}
where $pdz^2= \left\langle -\nabla_{\frac{\partial}{\partial_z}}N,\frac{\partial}{\partial_z}   \right\rangle dz^2$ is the Hopf differential of the surface.

Now, we can rewrite the equations for the surface
$\Sigma$ in terms of the conformal parameter $z$ induced by $I$, generalizing what was done in \cite{AEG} and \cite{AEG2}.

\begin{lemma}\label{hl1}
Let $\varphi:\Sigma \looparrowright \mrf$ be an isometric immersion from a complete orientable surface with a mean curvature function $H$. Then, for a conformal parameter $z$ for the first fundamental form, the following equations hold:
\begin{eqnarray}
\label{he4} p_{\overline{z}} &= & \dfrac{\lambda}{2}H_z +\dfrac{\lambda}{2}\left( f^{\prime\prime}+\kappa e^{-2f}\right)\nu h_z \\
\label{he5} \vert\vert \nabla h\vert\vert^2 & = & \dfrac{4}{\lambda}\vert h_z\vert^2 \\
\label{he6} \nu_z & = & -Hh_z -\dfrac{2p}{\lambda}h_{\overline{z}} -f^\prime \nu h_z \\
\label{he7} h_{z\overline{z}} &= & \nu\dfrac{\lambda H}{2}+ f^{\prime}\left[ \dfrac{\lambda}{2}-\vert h_z\vert^2\right] \\
\label{he8} \nu_{z\overline{z}} & = &  -\left( H_{\overline{z}} h_z+H_zh_{\overline{z}} \right) -\dfrac{\lambda\nu}{4}\left[ \left( f^{\prime\prime}+\kappa e^{-2f}\right)(1-\nu^2) +\dfrac{8\vert p\vert^2}{\lambda^2} +2H^2  \right] \\
\nonumber &  & + f^\prime \left[ \dfrac{2p}{\lambda} h_{\overline{z}}^2-H\left(\dfrac{\lambda}{2}-\vert h_z\vert^2 \right) - \left(\nu_{\overline{z}}h_z+\nu h_{z\overline{z}} \right) \right] -f^{\prime\prime}\nu\vert h_z\vert^2
\end{eqnarray}

\end{lemma}

\section{ Geometrical Applications to theory of surfaces of constant curvature in   \texorpdfstring{$\mathbb{R} \times_{f} \mathbb{R}^{2}$}.}

In this section , we apply the necessary equations calculated in the previous section to obtain some consequences for surfaces with constant mean curvature or constant extrinsic mean curvature in $\mathbb{R} \times_{f} \mathbb{R}^{2}$. Before that, we recall the following result: 

\begin{lemma}
Let $(M,g)$ be a Riemannian manifold a suppose that $\hat{g} = cg$ where $c$ a postie constant, then $\Delta_{\hat{g}} = \frac{1}{c} \Delta_{g}$  
\end{lemma} 

In order to prove our first result, we establish the following technical lemma, which is valid for surfaces with positive constant extrinsic curvature immersed in $\mathbb{R} \times_{f} \mathbb{R}^{2}$. 

\begin{lemma}
\label{13}
Let  $\varphi:\Sigma \looparrowright \rrf$ be an immersion from a complete, orientable surface, having constant positive extrinsic curvature $K_e$ and consider a conformal parameter $z$ for the second fundamental form $II$. Then, the function $g(\nu) = \frac{1}{\sqrt{K_{e}}}e^{f}\nu$ satisfies

\begin{equation}
\label{eq16g}
g_{z\overline{z}} = \frac{e^{f}}{\sqrt{K_{e}}} \Big( -K_{e}F\nu + \frac{K_{e}f'D}{\rho} \Big)
\end{equation}
\end{lemma}

%We recall that for the function $t: \ee \rightarrow \mathbb{R}$, we have $t|_{\Sigma} = h$ 

\begin{proof}
First, $\sqrt{K_{e}}g_{z} = e^{f} \big( f'h_{z} \nu + \nu_{z} \big)$. Since $(fh_{z})_{\overline{z}} = f''|h_{z}|^{2} + f'h_{z\overline{z}}$ and from the formula for $\nu_{z \overline{z}}$ in equation \eqref{e15}, we have the following:

\begin{eqnarray*}
    \sqrt{K_{e}}g_{z\olz} & = & e^{f}f'h_{\olz} \big( f'h_{z} \nu + \nu_{z} \big) + e^{f} \big( ( f''|h_{z}|^{2} + f'h_{z\overline{z}}) \nu + f'h_{z}\nu_{\olz} + \nu_{z\olz}\big) \\
    & = & e^{f} \Big[] (f')^{2} \nu |h_{z}|^{2} + f'h_{\olz} \nu_{z} + f'' \nu |h_{z}|^{2} + f'\nu h_{z\olz} + f'h_{z}\nu_{\olz}  - K_{e} F \nu - \frac{K_{e}f^{\prime}}{\rho} (\alpha h_{z}  - D ) \\
    & & - f^{\prime} (\nu_{z} h_{\overline{z}} + \nu h_{z\overline{z}}) - f^{''}\nu \abs{h_{z}}^{2}\Big] \\
    & = &  e^{f} \Big( (f')^{2} \nu |h_{z}|^{2} +  f'h_{z}\nu_{\olz}  - K_{e} F \nu - \frac{K_{e}f^{\prime}}{\rho} (\alpha h_{z}  - D )  \Big).
\end{eqnarray*}
Finally, from equation \eqref{e11},  $\nu_{\olz}  =  \dfrac{\alpha K_e}{\rho} -f^\prime\nu h_{\olz}$. Then:

\begin{eqnarray*}
    \sqrt{K_{e}}g_{z\olz} & = & e^{f} \Big( (f')^{2} \nu |h_{z}|^{2} +  f'h_{z}\Big(  \dfrac{\alpha K_e}{\rho} -f^\prime\nu h_{\olz} \Big)  - K_{e} F \nu - \frac{K_{e}f^{\prime}}{\rho} \Big(\alpha h_{z}  - D \Big)  \Big) \\
    & = &  e^{f} \Big(  - K_{e} F \nu  +  + \frac{K_{e}f'D}{\rho}  \Big)
\end{eqnarray*}
and lemma follows.   
\end{proof}

Our first result is a height estimate for surfaces with constant positive extrinsic curvature immersed in the warped product $\rrf$, under certain assumptions about the warping function $f$. This result, in particular, extends the result obtained in \cite{EGR} for this class of surfaces immersed in some product spaces.

\begin{theorem} 
\label{HET}
Let $\Sigma \looparrowright \mathbb{R} \times_{f} \mathbb{R}^{2} $  be a compact graph in a domain of $\mathbb{R}^{2}$ where $f$ is a non-negative warping function, such that the first derivative $f'$ is non-positive and the second derivative $f''$ is non-negative. Suppose that $\Sigma$  has a positive constant extrinsic curvature $K_{e}$ and the boundary $\partial \Sigma$ is a subset of the slice $\{ 0 \} \times \mathbb{R}^{2}$. Then, the height function $h: \Sigma \rightarrow \mathbb{R}$ satisfies
\begin{equation*}
    h(p) \leq \frac{e^{f(0)}}{\sqrt{K_{e}}}
\end{equation*}
for each $p \in \Sigma$.  
\end{theorem}

\begin{proof} 
We compute the Laplacian of the function $h + g(v)$ with respect to the parameter $z$ induced by the second fundamental form $II$, where $g(\nu) = \frac{1}{\sqrt{K_{e}}} e^{f}\nu$. Then, from equation \eqref{e13} and lemma \ref{13}, we get 

\begin{eqnarray*}
    \Big( h + g(v) \Big)_{z\olz} & = & \rho \nu \Big( 1 - \frac{ f^{\prime\prime}(1-\nu^{2})}{2 K_{e}} \Big) - f^{\prime} (F - h_{\overline{z}}h_{z}) + \frac{e^{f}}{\sqrt{K_{e}}} \Big( -K_{e}F\nu + \frac{K_{e}f'D}{\rho} \Big)\\
    & = & \rho \nu - \sqrt{K_{e}} e^{f} F \nu - \frac{ \rho \nu f^{\prime\prime}(1-\nu^{2})}{2 K_{e}} - f^{\prime} (F - \vert h_{\overline{z}} \vert^{2}) + \frac{e^{f} \sqrt{K_{e}} f' D}{\rho}.
\end{eqnarray*}
We proceed to estimate the last equality. By the one hand:

\begin{itemize}
    \item Since $\nu \leq 0$, $f'' \geq 0$ and $K_{e} >0$, then $- \frac{ \rho \nu f^{\prime\prime}(1-\nu^{2})}{2 K_{e}} \geq 0$,
    \item $f' \leq 0$ and $D < 0$, then  $\frac{e^{f} \sqrt{K_{e}} f' D}{\rho} \geq 0$ and 
    \item from equation \eqref{e9} and $F>0$, we get $-f^{\prime} (F - \vert h_{\overline{z}} \vert^{2}) = - f^{'} \big( F \nu^{2} - \frac{\vert \alpha \vert^{2}}{D} \big) \geq 0$.
\end{itemize}
On the other hand, from equation \eqref{e2}, $\rho = \sqrt{-D} \sqrt{K_{e}}$, hence

\begin{equation*}
    \rho \nu - \sqrt{K_{e}} e^{f} F \nu  = \sqrt{K_{e}} \nu \big( \sqrt{-D} - e^{f} F\big)  =  \sqrt{K_{e}} \nu \big( \sqrt{F^{2} - \vert E \vert^{2}} - e^{f} F \big)
     \geq  \sqrt{K_{e}} \nu F \big( 1 - e^{f} ) \geq 0
\end{equation*}
The last inequality follows from the fact that $f\geq 0$ and $\nu \leq 0 $. Therefore, $\Delta_{II} (h + g(\nu)) \geq 0$ on $\Sigma$ but $h+g(\nu) \leq 0$ on $\partial \Sigma$. So, by the maximum principle $h + g(\nu) \leq 0$ in $\Sigma$, therefore  $h(p) \leq \frac{e^{f(0)}}{\sqrt{K_{e}}}$ for each point $p \in \Sigma$.
\end{proof}

The second result is a height estimate for surfaces with constant mean curvature immersed in $\rrf$, for which we give a lemma similar to the lemma \ref{13}. 

\begin{lemma}
\label{15}
Let  $\varphi:\Sigma \looparrowright \rrf$ be an immersion from a complete, orientable surface with constant mean curvature $H$ and consider a conformal parameter $z$ for the second fundamental form $I$, then for the function $g(\nu) = \frac{1}{H}e^{f}\nu$, we have the following 

\begin{equation}
\label{eqle2}
g_{z\overline{z}} = \frac{e^{f}}{H}\Big( \frac{-\lambda \nu}{4}\Big( f^{\prime\prime}(1-\nu^2)  + \frac{8 \vert p \vert^{2} }{\lambda^{2}} +2H^{2} \Big) - f^{\prime}H\frac{\lambda}{2} \Big)
\end{equation}
\end{lemma}

The previous lemma allows us to estimate the height of surfaces with constant mean curvature in the warped product $\mathbb{R} \times_{f} \mathbb{R}^{2}$. This estimate is similar to the one derived in the case of constant positive extrinsic curvature.

\begin{theorem}{{\bf Height estimates for CMC immersed surfaces in $
\mathbb{R} \times_{f} \mathbb{R}^{2}$.} } 
\label{HET2}
Let $\Sigma \looparrowright \mathbb{R} \times_{f} \mathbb{R}^{2} $  be a compact graph on a domain of $\mathbb{R}^{2}$ where $f$ is a non-negative warping function, such that $f'$ is non-positive and $f''$ is non-negative. Moreover, suppose that $\Sigma$ has a positive constant mean curvature $H$ and that the boundary $\partial \Sigma$ is contained in the slice $\mathbb{R}^{2} \times \{0 \}$. Then, the height function $h: \Sigma \rightarrow \mathbb{R}$ satisfies
\begin{equation*}
h(p) \leq  \frac{e^{f(0)}}{H}
\end{equation*}
for each $p \in \Sigma$.
\end{theorem}

\begin{remark}
We point out that Theorems \ref{HET} and \ref{HET2} were proved in \cite{AEG} and \cite{EGR} when the warping function
f is a constant function.
\end{remark}

Finally, we classify compact minimal graphs when the boundary is contained in a slice, with a certain restriction on the warping function $f$.

\begin{theorem}
Let $\Sigma \looparrowright \mathbb{R} \times_{f} \mathbb{R}^{2} $  be a compact minimal graph defined in a domain of $\mathbb{R}^{2}$. Suppose that $f'$ is nonnegative and that the boundary $\partial \Sigma$ is contained in the slice $\mathbb{R}^{2} \times \{0 \}$. Then $\Sigma$ must be a slice $\{t=0\}$ when $f^\prime(0)=0$ otherwise there is no such minimal surface.
\end{theorem}

\begin{proof}
From equations \eqref{he5} with $H = 0$, we have the following.

\begin{eqnarray*}
    h_{z \olz} & = & f'\Big( \frac{\lambda}{2} - \vert h_{z} \vert^{2} \Big) \\
    & = & f'\Big( \frac{\lambda}{2} - \frac{\lambda}{4}(1-\nu^{2})  \Big) \\
    & = & f'\frac{\lambda}{4} \Big(1+\nu^{2} \Big)
\end{eqnarray*}
hence $\Delta_{I}(h) = \frac{f'}{4} (1+\nu^{2}) \geq 0$, furthermore $h=0$ on $\partial \Sigma$. Since $h \geq 0$ on $\Sigma$, then, by the maximum principle, we have $h = 0$ on the whole $\Sigma$. So $\Sigma$ must be a slice.

\end{proof}

%%%%%%%%%%%%%%%%%%%%%%%%%%%%%%%%%
%%%%%%%%%%%%%%%%%%%%%%%%%%%%%%%%%
%%%%%%%%%%%%%%%%%%%%%%%%%%%%%%%%%
%%%%%%%%%%%%%%%%%%%%%%%%%%%%%%%%%


\begin{thebibliography}{24}

\bibitem{AR}
U. Abresch and H. Rosenberg. A Hopf differential for constant mean curvature surfaces in $\mathbb{S}^{2} \times \mathbb{R}$ and $\mathbb{H}^{2} \times \mathbb{R}$. {\it Acta Mathematica}, 193(2003), 141-174 




\bibitem{AEG}
J. Aledo, J. Espinar and J. Galvez. Height estimates for surfaces with positive constant mean curvature in $\mathbf{M}^2\times\R$, \textit{Illinois Journal of Mathematics.}, 52-1(2006),  {203}--{211}. S 0019-2082

\bibitem{AEG2} 
J. Aledo, J. Espinar and J. Galvez. The Codazzi equation for surfaces. \textit{ Advances in Mathematics}, 224-6 (2010),  2511-2530.

\bibitem{BO}
R. Bishop and B. ONeill,  Manifolds of negative curvature,  \textit{Trans. Amer. Math. Soc.}, 145 (1969), {1}--{49}.https://doi.org/10.1090/s0002-9947-1969-0251664-4


\bibitem{CFP}
A. Cambraia Jr, A. Folha and C. Pe{\~n}afiel. Totally umbilical surfaces in the warped product $\mathbb{M}(\kappa)_{f} \times I$. {\it Pacific Journal of Mathematics}. 313-2(2021), 343-364.


\bibitem{DC}
M. Do Carmo, Riemannian geometry ,  \textit{Springer-Verlag}, 6(1992). https://doi.org/10.1007/978-1-4757-2201-7

\bibitem{ER}
J. Espinar and H. Rosenberg. Complete constant mean curvature surfaces and Bernstein type theorems in $\mathbb{M}^{2} \times \mathbb{R}$.
{\it J. Differential Geom}. 82(3), 2009, 611-628.

\bibitem{ER2}
J. Espinar and H. Rosenberg. Complete constant mean curvature surfaces in homogeneous spaces.{\it Comment. Math. Helv}. 86 (3), 2011, 659-674.




\bibitem{EGR}
 J. Espinar, J. G{\'a}lvez and H. Rosenberg, Complete surfaces with positive extrinsic curvature in product
 spaces,\textit{Comment. Math. Helv.}, 84(2009), {351}--{386}.

\bibitem{FPS}
A. Folha, C. Pe{\~n}afiel and W. Santos. Height estimates for $H$-surfaces in the warped product $M \times_{f} \mathbb{R}$,  {\it Ann Glob Anal Geom}, 55(2019), 55-67.


\bibitem{ES1}
M-A. Lawn, and M. Ortega.A fundamental theorem for hypersurfaces in semi-Riemannian warped products,  {\it Journal of Geometry and Physics}, 90(2015), 55-70.



\bibitem{GIR}
S. Garcia, D. Impera and M. Rigoli. A sharp height estimate for compact hypersurfaces with constant k-Mean curvature in warped product spaces. {\it Proceedings of the Edinburgh Mathematical Society}, 58-2(2015), 403-419.


\bibitem{Ho} H. Hopf. Differential geometry in the large, volume 1000. \emph{Springer-Verlag
Berlin Heidelberg}, 1983.

\bibitem{M}
F. Morabito. Height estimate for special Weingarten surfaces of elliptic type in $\mathbb{M}^{2}(c) \times \mathbb{R}$. {\it Proceedings of the American Mathematical Society}. 1 (2014), 14-22.

\bibitem{TK1}
T.K. Milnor. Codazzi pairs on surfaces. {\it Global Differential Geometry and Global Analysis}.  Lecture Notes in Mathematics, vol 838 (2006). 263-274.




\end{thebibliography}
\end{document}